\documentclass[a4paper,10pt]{amsart}
\usepackage[utf8]{inputenc}
\usepackage[pdftex]{graphicx}
\usepackage{a4wide}
\usepackage[english]{babel}
\usepackage[T1]{fontenc}
\usepackage{amsmath,amssymb,amsthm,stmaryrd}
\usepackage{mathrsfs}
\usepackage{esint}
\usepackage{mathtools}
\usepackage{comment}
\usepackage{extarrows}
\usepackage{color}

\newcommand{\R}{\mathbb{R}}

\newcommand{\Q}{\mathscr{Q}}
\newcommand{\N}{\mathbb{N}}
\newcommand{\Z}{\mathbb{Z}}

\newcommand{\eps}{\varepsilon}
\newcommand{\ud}[0]{\,\mathrm{d}}

\newcommand{\A}{\mathscr{A}}
\newcommand{\D}{\mathscr{D}}

\newtheorem{theorem}[equation]{Theorem}
\newtheorem*{theorem*}{Theorem}
\newtheorem{lemma}[equation]{Lemma}

\newtheorem{corollary}[equation]{Corollary}
\newtheorem{proposition}[equation]{Proposition}

\theoremstyle{definition}
\newtheorem{defin}[equation]{Definition}
\newtheorem*{defin*}{Definition}
\newtheorem{example}[equation]{Example}
\newtheorem{remark}[equation]{Remark}
\newtheorem*{remark*}{Remark}
\newcommand{\rref}[1]{$\left(\ref{#1}\right)$}
\numberwithin{equation}{section}

\title{Weak $A_\infty$ weights and weak Reverse H\"older property in a space of homogeneous type}

\author{Theresa C. Anderson, Tuomas Hyt\"onen \and Olli Tapiola}

\address{Theresa C. Anderson, Department of Mathematics, Brown University, Providence, RI 02912, USA}
\email{theresa\_anderson@brown.edu}

\address{Tuomas Hytönen, Department of Mathematics and Statistics, P.O.B. 68 (Gustaf H\"allstr\"omin katu 2b), FI-00014 University of Helsinki, Finland}
\email{tuomas.hytonen@helsinki.fi}

\address{Olli Tapiola, Department of Mathematics and Statistics, P.O.B. 68 (Gustaf H\"allstr\"omin katu 2b), FI-00014 University of Helsinki, Finland}
\email{olli.tapiola@helsinki.fi}

\date{November 24, 2014}
\thanks{T.A. is supported by an NSF graduate student fellowship.}
\thanks{T.H. and O.T. are supported by the European Union through the ERC Starting Grant 278558 ``Analytic-probabilistic methods for borderline singular integrals''.
They are also part of Finnish Centre of Excellence in Analysis and Dynamics Research.}

\begin{document}

\begin{abstract}
  In the Euclidean setting, the Fujii-Wilson-type $A_\infty$ weights satisfy a Reverse Hölder Inequality (RHI) but in spaces of homogeneous 
  type the best known result has been that $A_\infty$ weights satisfy only a weak Reverse Hölder Inequality. In this paper, we compliment the results 
  of Hytönen, Pérez and Rela and show that there exist both $A_\infty$ weights that do not satisfy an 
  RHI and a genuinely weaker weight class that still satisfies a weak RHI. We also show that all the weights that satisfy a 
  weak RHI have a self-improving property but the self-improving property of the strong Reverse Hölder weights fails in a 
  general space of homogeneous type. We prove most of these purely non-dyadic results using convenient 
  dyadic systems and techniques.
\end{abstract}

\maketitle

\section{Introduction}

The relationship between the $A_\infty$ class and the Reverse Hölder Inequality (RHI) is well-known in the Euclidean setting: $w \in A_\infty$ if and only if
$w$ satisfies an RHI for some exponent $q > 1$ \cite{stein, hytonenperez, hytonenperezrela}. In a more general setting the 
results have not been as satisfactory. The following \emph{weak Reverse Hölder Inequality} has been the best result of the previous
type for Fujii-Wilson-type $A_\infty$ weights (for definition, see Section \ref{subsection:weight_classes}) in an arbitrary space of homogeneous type $(X,\rho,\mu)$:
\begin{theorem*}[{\cite[Theorem 1.1]{hytonenperezrela}}]
  For any $w \in A_\infty$ we have
  \begin{align*}
    \left( \fint_B w^{r(w)} \, d\mu \right)^{1/r(w)} \le C_X \fint_{2\kappa B} w \, d\mu
  \end{align*}
  for any ball $B$, where $\kappa$ is the quasi-triangular constant of $\rho$, $r(w) \coloneqq 1 + \frac{c_X}{[w]_{A_\infty}}$ and constants $c_X$ and $C_X$ depend only on $\kappa$
  and the doubling constant of $\mu$.
\end{theorem*}
The previous estimate is called a weak inequality because of the dilation of the ball on the right hand side. 
Although there are no dilations of balls involved in the definition of the $A_\infty$ class, the theorem holds 
only in the weak form even in a purely metric setting where $\kappa = 1$. This leads to two natural questions:
\begin{enumerate}
  \item[1)] Does a strong RHI hold for $A_\infty$ weights in an arbitrary space of homogeneous type?
  \item[2)] Does a weak RHI hold for some weight class that is genuinely weaker than the $A_\infty$ class?
\end{enumerate}

In this paper we answer these questions and explore some further questions related to question 2. In 
Sections \ref{section:weak_weight_classes}, \ref{section:dyadic_weight_classes} and \ref{section:A_infty_wRHI} we introduce the class 
of Fujii-Wilson-type weak $A_\infty$ weights and show that these weak weights satisfy a weak RHI. 
In Section \ref{section:gehring_lemma} we show that all the classes of weak Reverse Hölder weights have a self-improving property, and in Section \ref{section:failure_gehring},
we construct counterexamples that show that both a strong Reverse Hölder Inequality for $A_\infty$ weights and a self-improving property for strong Reverse Hölder weights fail 
in general spaces of homogeneous type.

In \cite{hytonenperezrela}, the authors proved their weak RHI theorem using a Calderón-Zygmund decomposition and working directly with
balls and their dilations. Although similar arguments would be valid in our situation as well, we deliberately take a different 
approach. We actively use the results of the second author and A. Kairema \cite{hytonenkairema}
and take several adjacent Christ-type dyadic systems \cite{christ} to give our weight class an alternative characterization. 
This characterization allows us to follow the elegant proof of the Euclidean ``$A_\infty \Rightarrow \text{RHI}$'' theorem of \cite{hytonenperezrela}.
Similarly, we give alternative ``dyadic'' characterizations to the different classes of weak Reverse Hölder weights and use dyadic arguments to prove 
their self-improving property.

The main reason for taking the dyadic approach is that it allows us to use dyadic cubes in all our decomposition arguments instead of finding a suitable 
Caldéron-Zygmund decomposition for each proof. This way our proofs become both shorter and more straightforward. Although we use several adjacent dyadic 
systems instead of a single one, the Euclidean dyadic techniques still translate well to our setting; basically, we only need to take some additional 
localization arguments into consideration. Also, this way we can prove the self-improving property of the weak Reverse Hölder weight classes as an application of 
the weak Reverse Hölder property of the weak $A_\infty$ weights.

Different types of weak Reverse Hölder Inequalities have an important role in the theory of partial differential 
equations and they appear in the literature frequently (see e.g. \cite{stredulinsky, heinonenkilpelainenmartio, kinnunen}). 
Weak $A_\infty$ weights, on the other hand, are much less common but they have been used in some classical and recent articles
related to analysis in $\R^n$. For example, in the early 1980's, E. Sawyer used these weights to give a sufficient condition 
on weights for a certain two-weight norm inequality to hold \cite{sawyer}, and more recently, S. Hofmann and J. M. Martell have used a 
weak $A_\infty$ condition to characterize certain metric properties of the harmonic measure \cite{hofmannmartell}. Both in \cite{sawyer} and \cite{hofmannmartell},
the weak $A_\infty$ condition is different than the condition we primarily use in this paper, but in Section \ref{section:equivalent_definitions} 
we show that these two definitions agree in spaces of homogeneous type.

We conclude the introduction by noting that our results already turned out to be useful in a recent paper by K. Li \cite{li}.

\section{Set-up and notation}
Let $(X,\rho,\mu)$ be a space of homogeneous type in the sense of Coifman and Weiss, i.e, $(X,\rho)$ is a quasi-metric space equipped with a doubling Borel measure $\mu$.
That is, $\rho$ satisfies the axioms of a metric except for the triangle inequality, which holds in the weaker form
\begin{align*}
  \rho(x,y) \le \kappa(\rho(x,z) + \rho(z,y))
\end{align*}
for some $\kappa \ge 1$, and there exists a constant $D \coloneqq D_\mu$ such that
\begin{align*}
  \mu(B(x,2r)) \le D \mu(B(x,r))
\end{align*}
for every ball $B(x,r) \coloneqq \{y \in X \colon \rho(x,y) < r\}$. The inequality above makes sense if balls are Borel sets which is not always true in the 
quasimetric case. Thus, for simplicity, we will assume that all balls are Borel sets but we note that, up to changing constants throughout, it is possible to eliminate 
this assumption by \cite[Section 5]{auscherhytonen} (see also the classical results of R. Macías and C. Segovia related to this topic \cite{maciassegovia}).
As usual, the dilation of a ball $B \coloneqq B(x,r)$ will be denoted by $\lambda B \coloneqq B(x,\lambda r)$ for 
every $\lambda > 0$.

We do not track the dependencies of our bounds on the structural constants (i.e. the constants depending only on $\kappa$ and $D$). The reason for this is simply that 
in many proofs the structural constants become rather complicated. Thus, for clarity, we use often the notation $E \lesssim F$ if $E \le cF$ for some structural constant $c$
and $E \lesssim_\alpha F$ if $E \le cF$ for some constant $c$ depending on $D$, $\kappa$ and $\alpha$. If $E \lesssim F \lesssim E$, we denote $E \eqsim F$.

The doubling property of $\mu$ implies the following \emph{geometrical doubling property} of $\rho$: any ball $B(x,r)$ can be covered by at most $N \coloneqq N_{D,\kappa}$,
balls of radius $r/2$ (it is not difficult to show that $N \le D^{6 + 3\log_2 \kappa}$).
Furthermore, for any $\varepsilon \in (0,1]$, the ball $B(x,r)$ can be covered by at most $N_\varepsilon \coloneqq N \varepsilon^{-\log_2 N}$ 
balls of radius $\varepsilon r$.

For every $\mu$-measurable set $E$ and for every measurable function $f$ we denote
\begin{align*}
  f(E) \coloneqq \int_E f \, d\mu \ \ \ \text{ and } \ \ \ f_E \coloneqq \fint_E f \, d\mu = \frac{1}{\mu(E)} \int_E f \, d\mu.
\end{align*}
Almost all of the functions in our paper will be weights, i.e. nonnegative measurable functions.

\subsection{$A_\infty$ and $RH_q$ classes}
\label{subsection:weight_classes}

In the Euclidean space $\R^n$, there are several equivalent ways to define the $A_\infty$ class \cite{stein,hruscev}. In general, the equivalence breaks down in an 
arbitrary space of homogeneous type \cite{strombergtorchinsky,kortekansanen}. The starting point of our investigation is the following definition used by the second 
author, Pérez and Rela in a space of homogeneous type \cite{hytonenperezrela} and extending the earlier Euclidean notion by Fujii \cite{fujii} and
Wilson \cite{wilson}: a weight $w$ belongs to the (strong) $A_\infty$ class if
\begin{align*}
  [w]_{A_\infty} \coloneqq \sup_B \frac{1}{w(B)} \int_B M(1_B w) \, d\mu < \infty,
\end{align*}
where the supremum is taken over all the balls $B$ in $X$ and $M$ is the non-centered Hardy-Littlewood maximal operator
\begin{align*}
  Mf(x) \coloneqq \sup_{B \ni x} \fint_B |f| \, d\mu.
\end{align*}

For every $q \in (1,\infty)$ we say that a weight $w$ belongs to the 
(strong) $q$-Reverse Hölder class $RH_q$ if there exists a finite constant $[w]_{RH_q}$ such that
\begin{align*}
  \left( \fint_B w^q \, d\mu \right)^{1/q} \le [w]_{RH_q} \fint_B w \, d\mu
\end{align*}
for every ball $B$. Here $[w]_{RH_q}$ denotes the smallest of the constants that satisfy the condition above.

\section{Weak $A_\infty$ and weak Reverse Hölder classes}
\label{section:weak_weight_classes}

Our goal in the next sections is to prove that we can weaken the $A_\infty$ class in such a way that the functions of the new, strictly bigger class still satisfy a 
weak RHI. For this we also define the class of weak Reverse Hölder weights.

\subsection{Weak $A_\infty$}

Since dilations of balls are present in the weak Reverse Hölder Inequalities, the following definition is natural:

\begin{defin}
  \label{definition:sigma-weak-A-infty}
  For every $\sigma \ge 1$ we say that a weight $w$ belongs to the \emph{$\sigma$-weak $A_\infty$ class} $A_\infty^\sigma$ if 
  \begin{align*}
    [w]_\infty^\sigma \coloneqq \sup_B \frac{1}{w(\sigma B)} \int_B M(1_B w) \, d\mu < \infty
  \end{align*}
  where the supremum is taken over all the balls $B \subseteq X$.
\end{defin}

This definition genuinely weakens the $A_\infty$ class. We can see this by a simple example.

\begin{example}
  Consider the weight $w$, $w(x) = e^x$, in $(\R,dx)$. Then for every interval $I = (a-r,a+r)$ we have
  \begin{align*}
    \int_I M(1_I w) \, dx \le |I|e^{a+r} = \frac{2}{\sigma-1} \cdot (\sigma-1)r e^{a+r} \le \frac{2}{\sigma-1} \int_{a+r}^{a+\sigma r} e^x \, dx \le \frac{2}{\sigma-1} w(\sigma I).
  \end{align*}
  In particular, $w \in A_\infty^\sigma$ for every $\sigma > 1$.
  
  On the other hand, for every $k \in \N$ and $I_k \coloneqq (k,3k)$ we have
  \begin{align*}
    \frac{w(2I_k)}{w(I_k)} = \frac{e^{4k}-1}{e^{3k} - e^k} \ge e^k \xlongrightarrow{k \to \infty} \infty. 
  \end{align*}
  Thus, the measure $w \, dx$ is not doubling so $w \notin A_\infty$.
\end{example}

Although the constants $[w]_\infty^\sigma$ depend on the parameter $\sigma$, the classes $A_\infty^\sigma$ contain 
the same functions if $\sigma > \kappa$:

\begin{theorem}
  \label{thm:A-infty-classes} $A_\infty^\sigma = A_\infty^{\sigma'}$ for every $\sigma,\sigma' > \kappa$. 
\end{theorem}

In the proof of Theorem \ref{thm:A-infty-classes} we use the following standard lemma.

\begin{lemma}
  Let $f$ be a locally integrable function. Then for every $x \in X$ it holds that
  \begin{align}
    \label{pointwise_maximal_functions} Mf(x) \lesssim M_cf(x),
  \end{align}
  where
  \begin{align*}
   M_cf(x) = \sup_{r > 0} \fint_{B(x,r)} |f| \, d\mu.
  \end{align*}
\end{lemma}

\begin{proof}
  Suppose that $x \in B(z,r)$, $r > 0$. Then $B(x,r) \subseteq B(z,2\kappa r)$ and $B(z,r) \subseteq B(x,2\kappa r)$. Thus, by the doubling property of $\mu$, we have
  \begin{align*}
    \frac{1}{\mu(B(z,r))} \int_{B(z,r)} |f| \, d\mu \lesssim \frac{1}{\mu(B(x,2\kappa r))} \int_{B(x,2\kappa r)} |f| \, d\mu \le M_cf(x),
  \end{align*}
  which proves the claim.
\end{proof}

\begin{proof}[Proof of Theorem \ref{thm:A-infty-classes}]
  Let $\sigma' < \sigma$. Clearly $A_\infty^{\sigma'} \subseteq A_\infty^\sigma$ so we only need to show that $A_\infty^{\sigma} \subseteq A_\infty^{\sigma'}$.
  
  Let $w \in A_\infty^{\sigma}$ and let $B := B(x_0,r)$ be any ball, $r > 0$. Then for every $y \in B$, $\varepsilon > 0$ and $z \in B(y,2\sigma\kappa\varepsilon r)$ we have
  \begin{align*}
    \rho(z,x_0) \le \kappa(\rho(z,y) + \rho(y,x_0)) < \kappa(2\sigma\kappa\varepsilon + 1)r.
  \end{align*}
  Thus, since $\sigma,\sigma' > \kappa$, for $\varepsilon := \frac{\sigma' - \kappa}{2\sigma \kappa^2} \in (0,1]$ we have
  $B(y,2\sigma\kappa \varepsilon r) \subseteq \sigma'B$ for every $y \in B$.
  
  By the geometrical doubling property of $\rho$, there is a finite set 
  $\{x_i \colon i = 1,2, \ldots, N_\varepsilon\} \subseteq B$ such that the balls $B_i := B(x_i,\varepsilon r)$ cover the ball $B$. Now
  \begin{align*}
    \int_B M(1_B w) \ &\le \ \sum_{i=1}^{N_\varepsilon} \int_{B_i} M(1_B w) \\ 
                      &\le \ \sum_{i=1}^{N_\varepsilon} \left( \int_{B_i} M(1_{2\kappa B_i} w) + \int_{B_i} M(1_{B \setminus 2\kappa B_i} w) \right) \\
                      &\eqqcolon \ \sum_{i=1}^{N_\varepsilon} \left( I_i + II_i \right).
  \end{align*}
  Notice that
  \begin{align}
    \label{I_i} I_i \le \int_{2\kappa B_i} M(1_{2\kappa B_i}w) \le [w]_\infty^{\sigma} w(2\sigma\kappa B_i) \le [w]_\infty^\sigma w(\sigma'B).
  \end{align}
  Also, we have $B(x,\varepsilon r) \subseteq 2\kappa B_i$ for every $x \in B_i$. Thus, for every $x \in B_i$ there holds
  \begin{align*}
    M(1_{B \setminus 2\kappa B_i}w)(x) \ \overset{\text{\rref{pointwise_maximal_functions}}}{\lesssim} \ M_c(1_{B \setminus 2\kappa B_i}w)(x) \
                                         &= \ \sup_{s \ge \varepsilon r} \frac{1}{\mu(B(x,s))} \int_{B(x,s)} 1_{B \setminus 2\kappa B_i}w \, d\mu \\
                                         &\le \ \frac{1}{\mu(B(x,\varepsilon r))} w(B) \\
                                         &\lesssim_{\varepsilon} \ \frac{1}{\mu(\sigma' B)} w(B).
  \end{align*}
  It follows that
  \begin{align}
    \label{II_i} II_i \lesssim_\varepsilon \frac{\mu(B_i)}{\mu(\sigma' B)} w(B) \overset{B_i \subseteq \sigma'B}{\le} w(\sigma'B).
  \end{align}
  Hence,
  \begin{align}
    \label{w-constants:1}  \int_B M(1_Bw) \overset{\text{\rref{I_i},\rref{II_i}}}{\le} N_\varepsilon( [w]_\infty^{\sigma} + C_{X,\varepsilon}) w(\sigma'B),
  \end{align}
  where the finite constant $N_\varepsilon( [w]_\infty^{\sigma} + C_{X,\varepsilon})$ depends on $w$, $\sigma$, $\sigma'$, $D$ and $\kappa$.
\end{proof}
  
Hence, it does not matter how much we weaken the $A_\infty$ class since the new class will always contain the same functions. Also, the next result shows that it is 
easy to compare the different weak $A_\infty$ constants with each other.

\begin{proposition}
  \label{prop:A_infty-constants}
  If $\sigma > \sigma' > \kappa$, then
  \begin{align*}
    [w]_\infty^\sigma \le [w]_\infty^{\sigma'} \lesssim_{\sigma,\sigma'} [w]_\infty^\sigma.
  \end{align*}
\end{proposition}

The proof of the right inequality is based on the following lemma:

\begin{lemma}
  \label{lemma:doubling_balls}
  Let $\sigma > 1$. Then for $\mu$-a.e. $x \in X$ there exist balls $B$ centered at $x$ with arbitrarily small radii such that
  \begin{align}
    \label{doubling_ball_condition} \mu(\sigma B) \le 2\sigma^{\log_2 N} \mu(B).
  \end{align}
\end{lemma}

\begin{remark}
  Lemma \ref{lemma:doubling_balls} is a quasimetric generalization of \cite[Lemma 3.3]{hytonenframework} (which in turn is based on earlier results of X. Tolsa \cite{tolsa}) for the choices $\alpha = \sigma$ and $\beta = 2\sigma^{\log_2 N}$. The structure of 
  our proof follows the structure of the original proof.
\end{remark}

\begin{proof}[Proof of Lemma \ref{lemma:doubling_balls}]
  Let us fix a ball $B \coloneqq B(x_0,r)$. For every $x \in B$ and $k \in \N$, denote $B_x^k \coloneqq B(x,\sigma^{-k}r)$. We call the point $x$ \emph{$k$-bad} if none
  of the balls $\sigma^j B_x^k$, $j = 0,1,\ldots,k$, satisfy the condition \rref{doubling_ball_condition}. Since $\sigma^k B_x^k = B(x,r) \subseteq 2\kappa B$, for every 
  $k$-bad point $x$ we have
  \begin{align*}
    \mu(B_x^k) \ \le \ \frac{1}{(2\sigma^{\log_2 N})^k} \mu(\sigma^k B_x^k) 
               \ \le \ \frac{1}{(2\sigma^{\log_2 N})^k} \mu(2\kappa B).
  \end{align*}
  Let $Y$ be a maximal $\sigma^{-k}r$-separated family among the $k$-bad points of $B$. Then it holds that $\{x \in B \colon x \text{ is } k\text{-bad}\} \subseteq \bigcup_{y \in Y} B_y^k$.
  Since the balls $\frac{1}{2\kappa}B_y^k$ are disjoint and their centres are contained in $B$, the geometrical doubling property of $\rho$ implies that $|Y| \le N(2\kappa\sigma^k)^{\log_2 N}$.
  Thus,
  \begin{align*}
    \mu(\{x \in B \colon x \text{ is } k\text{-bad}\}) \ \le \ \sum_{y \in Y} \mu(B_y^k) \ &\le \ \sum_{y \in Y} \frac{1}{(2\sigma^{\log_2 N})^k} \mu(2\kappa B) \\
                                                                                         &\le \ \frac{N(2\kappa\sigma^k)^{\log_2 N}}{(2\sigma^{\log_2 N})^k} \mu(2\kappa B) \\
                                                                                         &= \ N^{2+\log_2 \kappa} \cdot 2^{-k} \mu(2\kappa B) \longrightarrow 0
  \end{align*}
  as $k \to \infty$. Hence, $\mu(\{x \in B \colon x \text{ is } k\text{-bad for every } k \in \N\} = 0$ for any ball $B$.
  
  Since $X = \bigcup_{k=1}^\infty B(x_0,k)$, it follows that $\mu(\{x \in X \colon x \text{ is } k\text{-bad for every } k \in \N\} = 0$, which proves the claim.
\end{proof}

\begin{proof}[Proof of Proposition \ref{prop:A_infty-constants}]
  Since for every ball $B$ it holds that
  \begin{align*}
    \int_B M(1_B w) \, d\mu \ \le \ [w]_\infty^{\sigma'} w(\sigma'B) \ \le \ [w]_\infty^{\sigma'} w(\sigma B),
  \end{align*}
  we know that $[w]_\infty^\sigma \le [w]_\infty^{\sigma'}$.

  By \rref{w-constants:1}, we know that $[w]_\infty^{\sigma'} \le N_\varepsilon( [w]_\infty^{\sigma} + C_{X,\sigma,\sigma'})$ for $\sigma > \sigma' > 1$, 
  $\varepsilon = \frac{\sigma' - \kappa}{2\sigma \kappa^2}$.
  On the other hand, by 
  Lemma \ref{lemma:doubling_balls}, we know that there exists $B$ such that $w(\sigma B) \le 2\sigma^{\log_2 N} w(B)$. Thus,
  \begin{align*}
    \frac{1}{2\sigma^{\log_2 N}} w(\sigma B) \ \le \ w(B) \ \le \ \int_B M(1_B w) \, d\mu.
  \end{align*}
  In particular, $[w]_\infty^{\sigma} \ge 1/(2\sigma^{\log_2 N})$ and
  \begin{align*}
    [w]_\infty^{\sigma'} \ \le \ N_\varepsilon [w]_\infty^\sigma(1 + 2\sigma^{\log_2 N} C_{X,\sigma,\sigma'}) \ \le \ 4N_\varepsilon \sigma^{\log_2 N} C_{X,\sigma,\sigma'} [w]_\infty^\sigma.
  \end{align*}
\end{proof}

Theorem \ref{thm:A-infty-classes} and Proposition \ref{prop:A_infty-constants} give us now a natural way to define the weak $A_\infty$ class:

\begin{defin}
  A weight $w$ belongs to the \emph{weak $A_\infty$ class} $A_\infty^{\text{weak}}$ if $w \in A_\infty^\sigma$ for some $\sigma > \kappa$.
\end{defin}

\subsection{Weak Reverse Hölder classes}

We define the weak Reverse Hölder classes similarly as we defined the weak $A_\infty$ class.

\begin{defin}
For every $\sigma \ge 1$ and $q > 1$, a weight $w$ belongs to the \emph{$\sigma$-weak $q$-Reverse Hölder class} $RH_q^\sigma$ if there exists a finite constant $[w]_{RH_q}^\sigma$ such that
$w$ satisfies the $\sigma$-weak reverse Hölder property
\begin{align*}
  \left( \fint_B w^q \, d\mu \right)^{1/q} \ \le \ [w]_{RH_q}^\sigma \fint_{\sigma B} w \, d\mu
\end{align*}
for every ball $B$. Here $[w]_{RH_q}^\sigma$ is the smallest of the constants that satisfy the condition above.
\end{defin}
Again, although the definition seems to give us a different weight class depending on $\sigma$ for every $q > 1$, this is not the case:
\begin{theorem}
  \label{thm:wRHI-classes}
  $RH_q^\sigma = RH_q^{\sigma'}$ for every $\sigma,\sigma' > \kappa$.
\end{theorem}

\begin{proof}
  Let $\sigma' < \sigma$. Then for every $w \in RH_q^{\sigma'}$ we have
  \begin{align*}
    \left( \fint_B w^q \, d\mu \right)^{1/q} \ \le \ [w]_{RH_q}^{\sigma'} \fint_{\sigma' B} w \, d\mu \ \lesssim_{\sigma,\sigma'} \ [w]_{RH_q}^{\sigma'} \fint_{\sigma B} w \, d\mu.
  \end{align*}
  Thus, $RH_q^{\sigma'} \subseteq RH_q^\sigma$.
  
  Let then $w \in RH_q^\sigma$ and let $B := B(x_0,r)$ be any ball, $r > 0$. Then for every $y \in B$, $\varepsilon > 0$ and $z \in B(y,\sigma\varepsilon r)$ we have
  \begin{align*}
    \rho(z,x_0) \ \le \ \kappa(\rho(z,y) + \rho(y,x_0)) \ < \ \kappa(\sigma\varepsilon + 1)r.
  \end{align*}
  Since $\sigma,\sigma' > \kappa$, we can choose a constant $\varepsilon := \frac{\sigma' - \kappa}{\sigma\kappa} > 0$ such that $B(y,\sigma\varepsilon r) \subseteq \sigma' B$
  for every $y \in B$.
  
  By the geometrical doubling property of $\rho$, there exists a finite set $\{z_i \colon i=1,2,\ldots,N_\varepsilon\} \subseteq B$ such that the balls $B_i := B(z_i,\varepsilon r)$ cover
  the ball $B$. Since $B_i \subseteq \kappa(\varepsilon + 1)B$ and
  \begin{align*}
    \sigma'B \ \subseteq \ \frac{\kappa}{\sigma \varepsilon}(\sigma' + 1) \sigma B_i
  \end{align*}
  for every $i = 1,2,\ldots,N_\varepsilon$, we have
  \begin{align*}
    \fint_B w^q \, d\mu \ &\le \ \sum_{i=1}^{N_\varepsilon} \frac{1}{\mu(B)} \int_{B_i} w^q \, d\mu \\
                          &= \ \sum_{i=1}^{N_\varepsilon} \frac{\mu(B_i)}{\mu(B)} \fint_{B_i} w^q \, d\mu \\
                          &\le \ ([w]_{RH_q}^\sigma)^q \sum_{i=1}^{N_\varepsilon} \frac{\mu(B_i)}{\mu(B)} \left( \fint_{\sigma B_i} w \, d\mu \right)^q \\
                          &\le \ ([w]_{RH_q}^\sigma)^q \sum_{i=1}^{N_\varepsilon} \frac{\mu(B_i)}{\mu(B)} \left( \frac{\mu(\sigma'B)}{\mu(\sigma B_i)} \fint_{\sigma'B} w \, d\mu \right)^q 
                        \ \lesssim_{\sigma,\sigma',q} \ ([w]_{RH_p}^\sigma)^q \left( \fint_{\sigma'B} w \, d\mu \right)^q.
  \end{align*}
\end{proof}

The next result follows from the previous proof.

\begin{proposition}
  \label{prop:RH-constants}
  If $\sigma > \sigma' > \kappa$, then
  \begin{align*}
    [w]_{RH_q}^\sigma \lesssim_{\sigma,\sigma'} [w]_{RH_q}^{\sigma'} \lesssim_{\sigma,\sigma'} [w]_{RH_q}^\sigma.
  \end{align*}
\end{proposition}

Thus, we can define the weak Reverse Hölder classes the following way.

\begin{defin}
  A weight $w$ belongs to the \emph{weak $q$-Reverse Hölder class} $RH_q^{\text{weak}}$ if $w \in RH_q^\sigma$ for some $\sigma > \kappa$.
\end{defin}

\section{``Dyadic'' characterizations of weak $A_\infty$ and weak Reverse Hölder classes}
\label{section:dyadic_weight_classes}

Dyadic cubes in quasimetric spaces are objects that share many good properties with the usual Euclidean dyadic cubes but, unlike the Euclidean dyadic cubes, 
they are so abstract that it is usually very difficult to say which points actually belong to which cube. That is why it may seem strange to try to characterize 
some weight classes with respect to the dyadic cubes in spaces of homogeneous type. The motivation behind this is that if we can give suitable dyadic characterizations 
for the weight classes, we can use different Euclidean dyadic techniques also in our setting. These techniques make it fairly easy to prove 
some results related to the weak $A_\infty$ and weak Reverse Hölder classes.

We start by presenting some basic results related to the dyadic systems. We use quotation marks around the word dyadic if the definition in question uses several different dyadic systems instead of a single one.

\subsection{Dyadic systems in quasimetric spaces}

For our ``dyadic'' characterizations of $A_\infty^\text{weak}$ and $RH_q^\text{weak}$ we need to use several \emph{adjacent dyadic systems} so that for every ball $B$ we can find 
some cube $Q_B$ such that $B \subseteq Q_B$ and $\mu(B) \approx \mu(Q_B)$. For this we turn to the results of the
second author and Kairema \cite{hytonenkairema} (based on the ideas of the second author and H. Martikainen \cite{hytonenmartikainen}).

The first part of the following theorem is a version of the well-known dyadic constructions in spaces of homogeneous type of M. Christ \cite{christ} and E. Sawyer and R. L. Wheeden \cite{sawyerwheeden} 
with no measurability assumptions on the space $X$. The theorem holds in every geometrically doubling quasimetric space.

\begin{theorem}
  \label{thm:dyadic_systems}
  Suppose that the constant $\delta \in (0,1)$ satisfies $96\kappa^6 \delta \le 1$. Then there exist countable sets of points 
  $\{z_\alpha^{k,t} \colon \alpha \in \mathscr{A}_k\}$, $k \in \Z$, $t = 1,2,\ldots,K = K(\kappa,N,\delta)$, and a finite number of \emph{dyadic systems} 
  $\D^t \coloneqq \{Q_\alpha^{k,t} \colon k \in \Z, \alpha \in \A_k\}$, such that
  \begin{enumerate}
    \item[1)] for every $t \in \{1,2,\ldots,K\}$ and $k \in \Z$ we have
              \begin{enumerate}
                \item[i)] $X = \bigcup_{\alpha \in \A_k} Q_\alpha^{k,t}$ (disjoint union);
                \item[ii)] $Q,P \in \D^t \implies Q \cap P \in \{\emptyset,Q,P\}$;
                \item[iii)] $Q_\alpha^{k,t} \in \D^t \implies B(z_\alpha^{k,t},c_1\delta^k) \subseteq Q_\alpha^{k,t} \subseteq B(z_\alpha^{k,t},C_1\delta^k)$,
                            where $c_1 \coloneqq (12\kappa^4)^{-1}$ and $C_1 \coloneqq 4\kappa^2$;
              \end{enumerate}
    \item[2)] for every ball $B \coloneqq B(x,r)$ there exists a cube $Q_B \in \bigcup_t \D^t$ such that $B \subseteq Q_B$ and $\ell(Q_B) = \delta^{k-1}$, 
              where $k$ is the unique integer such that $\delta^{k+1} < r \le \delta^k$ and $\ell(Q_B) = \delta^{k-1}$ means that $Q_B = Q_\alpha^{k-1,t}$ for some indices $\alpha$ and $t$.
  \end{enumerate}
\end{theorem}

\begin{proof}
  See Theorem 4.1, the proof of Lemma 4.12, Remark 4.13 and Theorem 2.2 in \cite{hytonenkairema}.
  The choices $c_0 = (4\kappa^2)^{-1}$ and $C_0 = 2\kappa$ in \cite[Theorem 2.2]{hytonenkairema} are required for \cite[Theorem 4.1]{hytonenkairema} (see \cite[Lemma 4.10]{hytonenkairema}).
\end{proof}

We note that a stronger version of the second part of the previous theorem holds in $\R^n$ \cite[Lemma 2.5]{hytonenlaceyperez} and metric spaces 
\cite[Theorem 5.9]{hytonentapiola} but we do not need the additional properties introduced in those results.

\begin{remark}
  The dyadic cubes $Q_\alpha^{k,t}$ in Theorem \ref{thm:dyadic_systems} are Borel sets regardless of whether balls are Borel sets or not. This follows directly 
  from the construction; see \cite[Proposition 2.11, Lemma 2.18]{hytonenkairema}.
\end{remark}

From now on, let $\{\D^t \colon t = 1,2,\ldots,K\}$ be a fixed collection of adjacent dyadic systems given by Theorem \ref{thm:dyadic_systems}. Let us then denote
\begin{align}
  \label{dyadic_union} \D \coloneqq \bigcup_{t=1}^K \D^t \ \ \ \text{ and } \ \ \ \D^k \coloneqq \bigcup_{t=1}^K \D^{t,k} = \bigcup_{t=1}^K \{Q_\alpha^{k,t} \colon \alpha \in \A_k\}
\end{align}
for every $k \in \Z$. The collection $\D$ is not a dyadic system in the usual sense of the term but by Theorem \ref{thm:dyadic_systems}, the elements of $\D$ share some good properties of both balls 
and dyadic cubes.

Since the definitions of the weight classes in Section \ref{section:weak_weight_classes} involved balls and their dilations, we need a suitable way to enlarge the elements of $\D$ for our ``dyadic'' characterizations for those weight classes. We require that if we enlarge a cube $Q$, its enlargement needs to be another cube that contains all the nearby cubes of $Q$ of that same generation.

\begin{defin}
  Let $Q \in \D$. A \emph{generalized dyadic parent} (gdp) of $Q$ is any cube $Q^*$ such that $\ell(Q^*) = \frac{1}{\delta^2}\ell(Q)$ and
  for every $Q' \in \D$ such that $Q' \cap Q \neq \emptyset$ and $\ell(Q') = \ell(Q)$ we have $Q' \subseteq Q^*$.
\end{defin}

\begin{lemma}
  For every $Q \in \D$ there exists at least one gdp.
\end{lemma}

\begin{proof}
  Let $Q \coloneqq Q_\alpha^{k,t} \in \D$. If $Q' \cap Q \neq \emptyset$ and $\ell(Q') = \ell(Q)$, then by Theorem \ref{thm:dyadic_systems} we know that 
  $Q,Q' \subseteq B(z_\alpha^{k,t}, 3\kappa^2 C_1 \delta^k)$. Since $3\kappa^2 C_1 \delta^k \le \delta^{k-1}$, by Theorem \ref{thm:dyadic_systems} there exists 
  a cube $Q^* \in \D^{k-2}$ such that $B(z_\alpha^{k,t}, 3\kappa^2 C_1 \delta^k) \subseteq Q^*$.
\end{proof}

\subsection{Localized ``dyadic'' maximal operators and ``dyadic'' weak $A_\infty$}

Before we can formulate our ``dyadic'' characterization of the $A_\infty^\text{weak}$ class, we need to find a substitute for the localized Hardy-Littlewood maximal 
operator $1_B M(1_B \cdot)$ for every ball $B$. For this, let us denote
\begin{align*}
  \Q_Q \ &\coloneqq \ \{Q' \in \D \colon Q' \cap Q \neq \emptyset, \ell(Q') \le \ell(Q)\},\\
  \Q_Q^x \ &\coloneqq \ \{Q' \in \Q_Q \colon x \in Q'\},
\end{align*}
for every $Q \in \D$ and $x \in X$. It follows immediately that if $Q' \in \Q_Q$, then $Q' \subseteq Q^*$ and $(Q')^* \in \Q_{Q^*}$.

\begin{defin}
  For every $Q_0 \in \D$, we define the \emph{localized ``dyadic'' maximal operator} $M_{Q_0}$ with the following formula
  \begin{align*}
    M_{Q_0} f(x) \ = \ \sup_{Q \in \Q_{Q_0}^x} |f|_Q.
  \end{align*}
  We set $M_{Q_0} f(x) = 0$ if $\Q_{Q_0}^x = \emptyset$.
\end{defin}

We note that $\int_Q w \, d\mu \le \int_Q M_Q w \, d\mu$ for every $w$ and $Q$ by \cite[Proposition 4.5]{auscherhytonen}
(see also \cite[Theorem 6.2.4]{krantz} and \cite[Lemma 2.3]{andersonetall}).

In the proof of Lemma \ref{lemma:A_infty_dyadic} we see that we can always approximate Hardy-Littlewood maximal functions locally with some 
localized ``dyadic'' maximal functions. Thus, the localized ``dyadic'' maximal operator is strong enough for our purposes:

\begin{defin}
  We say that a weight belongs to the \emph{``dyadic'' weak $A_\infty$ class} $A_\infty^\D$ if
  \begin{align*}
    [w]_\infty^\D \ \coloneqq \ \sup_Q \inf_{Q^*} \frac{1}{w(Q^*)} \int M_Q w \, d\mu < \infty
  \end{align*}
  where the supremum is taken over all the cubes $Q \in \D$ and the infimum is taken over all the gdps of the cube $Q$.
\end{defin}

\begin{lemma}
  \label{lemma:A_infty_dyadic}
  $A_\infty^\D = A_\infty^\text{\emph{weak}}$.
\end{lemma}

\begin{proof} 
  Let $w \in A_\infty^\D$ and let $B$ be a ball. Then for some $k \in \Z$ we have $\delta^{k+2} < 2\kappa \cdot r(B) \le \delta^{k+1}$.
  Let $Q_0 \in \D$ be a cube such that $\ell(Q_0) = \delta^k$ and $B \subseteq Q_0$. Notice that if $x \in B$, then $B \subseteq B(x,2\kappa \cdot r(B))$. Thus,
  for every $x \in B$ we have
  \begin{align*}
    M(1_B w)(x) \overset{\text{\rref{pointwise_maximal_functions}}}{\lesssim} \sup_{r > 0} \, (1_B w)_{B(x,r)}
                = \sup_{r \le 2\kappa \cdot r(B)} \, (1_B w)_{B(x,r)}
                \lesssim \sup_{r \le 2\kappa \cdot r(B)} \, (1_B w)_{Q_{B(x,r)}}.
  \end{align*}
  Since $Q_{B(x,r)} \cap Q_0 \neq \emptyset$ and $\ell(Q_{B(x,r)}) \le \delta^k$, we know that $Q_{B(x,r)} \in \Q_{Q_0}^x$.
  In particular,
  \begin{align*}
    M(1_B w)(x) \lesssim \sup_{Q \in \Q_{Q_0}^x} (1_B w)_Q \le \sup_{Q \in \Q_{Q_0}^x} w_Q = M_{Q_0} w(x).
  \end{align*}
  Hence,
  \begin{align*}
    \int_B M(1_B w) \, d\mu \ \lesssim \ \int M_{Q_0}w \, d\mu \ \le \ [w]_\infty^\D w(Q_0^*) \ \lesssim \ [w]_\infty^\D w(\sigma B)
  \end{align*}
  for a large enough $\sigma > \kappa$ depending only on $\kappa$ and $D$.

  Let then $w \in A_\infty^{2\kappa}$, $Q_0 \coloneqq Q_\alpha^{k,t} \in \D$ and $x \in B \coloneqq B(z_\alpha^{k,t}, 3\kappa^2 C_1 \delta^k)$. Since $Q \subseteq B$ 
  for every $Q \in \Q_{Q_0}$, we have
  \begin{align*}
    M_{Q_0} w(x) = \sup_{Q \in \Q_{Q_0}^x} w_Q = \sup_{Q \in \Q_{Q_0}^x} (1_B w)_Q \lesssim M(1_B w)(x).
  \end{align*}
  Thus,
  \begin{align*}
    \int M_{Q_0} w \, d\mu \ = \ \int_B M_{Q_0} w \, d\mu \ \lesssim \ \int_B M(1_B w) \, d\mu \ \le \ [w]_\infty^{2\kappa} w(2\kappa B) \ \le \ [w]_\infty^{2\kappa} w(Q^*)
  \end{align*}
  for some $Q^*$ since $2\kappa \cdot r(B) = 24\kappa^5 \delta^k \le \delta^{k-1}$.
\end{proof}

The next result follows directly from the previous proof and Proposition \ref{prop:A_infty-constants}.

\begin{lemma}
  \label{lemma:A_constants}
  If $\sigma > \kappa$, then $[w]_\infty^\D \lesssim_\sigma [w]_\infty^\sigma$.
\end{lemma}

\begin{remark}
  \label{remark:dyadic_A_infty}
  By the proof of Lemma \ref{lemma:A_infty_dyadic} it is easy to see, that also weights $w$ that satisfy
  \begin{align*}
    \sup_{Q} \inf_{Q^{**}} \frac{1}{w(Q^{**})} \int_Q M_Q w \, d\mu \ < \ \infty
  \end{align*}
  belong to the class $A_\infty^\D$.
\end{remark}

\subsection{``Dyadic'' weak $RH_q$}

We prove the self-improving property of the $RH_q^\text{weak}$ classes in Section \ref{section:gehring_lemma} with the help of ``dyadic'' chracterizations of those classes:

\begin{defin}
  We say that a weight $w$ belongs to the \emph{``dyadic'' weak $RH_q$ class} $RH_q^\D$ if
  \begin{align*}
    [w]_{RH_q}^\D \coloneqq \sup_Q \inf_{Q^*} \frac{1}{w_{Q^*}} \left( \fint_{Q} w^q \, d\mu \right)^{1/q} \ < \ \infty,
  \end{align*}
  where the supremum is taken over all the cubes $Q \in \D$ and the infimum is taken over all the gdps of the cube $Q$.
\end{defin}

\begin{lemma}
  \label{lemma:dyadicRHq}
  $RH_q^\D = RH_q^{\text{\emph{weak}}}$ for every $q \in (1,\infty)$.
\end{lemma}

\begin{proof}
  Let $w \in RH_q^{2\kappa}$ and let $Q \coloneqq Q^{k,t} \in \D$. By Theorem \ref{thm:dyadic_systems}, we know that there exists a cube $Q^* \in \D^{k-2}$ such that 
  $Q \subseteq B(z^{k,t},C_1 \delta^k) \subseteq B(z^{k,t},2\kappa C_1 \delta^k) \subseteq Q^*$. Thus,
  \begin{align*}
    \left( \fint_Q w^q \, d\mu \right)^{1/q} \ &\lesssim \ \left( \fint_{B(z^{k,t},C_1 \delta^k)} w^q \, d\mu \right)^{1/q} \\
                                             \ &\le \ [w]_{RH_q}^{2\kappa}\fint_{2\kappa B(z^{k,t},C_1 \delta^k)} w \, d\mu
                                             \ \lesssim \ [w]_{RH_q}^{2\kappa} \fint_{Q^*} w \, d\mu,
  \end{align*}
  so $w \in RH_q^\D$.
  
  Let then $w \in RH_q^\D$ and let $B \coloneqq B(x,r)$ be any ball. Then
  \begin{align*}
    \left( \fint_B w^q \, d\mu \right)^{1/q} \lesssim \left( \fint_{Q_B} w^q \, d\mu \right)^{1/q}
                                             \le [w]_{RH_q}^\D \fint_{Q_B^*} w \, d\mu
                                             \lesssim [w]_{RH_q}^\D \fint_{\sigma B} w \, d\mu,
  \end{align*}
  for large enough $\sigma > \kappa$ independent of $B$.
\end{proof}

The next result follows directly from previous proof and Proposition \ref{prop:RH-constants}.

\begin{lemma}
  \label{lemma:dyadic_RH_q_constants}
  If $\sigma > \kappa$, then $[w]_{RH_q}^\D \lesssim_\sigma [w]_{RH_q}^\sigma \lesssim_\sigma [w]_{RH_q}^\D$.
\end{lemma}

\section{Weak Reverse Hölder property of the weak $A_\infty$ class}
\label{section:A_infty_wRHI}

Now we are ready to prove that the $A_\infty^{\text{weak}}$ weights satisfy a weak RHI. Instead of proving it directly, we show
that the weights of $A_\infty^\D$ satisfy a ``dyadic'' weak RHI which then implies the original claim. This way we can use the properties 
of the underlying dyadic structures and borrow some techniques from the Euclidean setting. For the following results, let $S \ge 1$ be a 
constant such that for all the cubes $Q_0 \in \D$ and $Q \in \Q_{Q_0}$, $\ell(Q) = \ell(Q_0)$, we have $\mu(Q_0^*) \le S \mu(Q)$.

Our poof follows closely the ideas and structure of the proof the Euclidean ``$A_\infty \Rightarrow \text{RHI}$'' theorem by the second author, Pérez and Rela 
\cite[Theorem 2.3]{hytonenperezrela}. Since the proof in question uses only one dyadic system, we need to take some additional arguments into consideration:

\begin{lemma}
  \label{lemma:A_infty_wRHI}
  Let $Q_0 \in \D$ and $\lambda \ge Sw_{Q_0^*}$. Then there exists a family of cubes $\{Q_j\}_j \subseteq \Q_{Q_0}$ such that 
  $\{M_{Q_0}w > \lambda\} = \bigcup_j Q_j$ and
  \begin{enumerate}
    \item[1)] $Q_j \subseteq Q_0^*$ for every $j$;
    \item[2)] $w_{Q_j} > \lambda$ and $w_Q \le \lambda$ for every $Q \in \D$ such that $Q_j \subsetneq Q$;
    \item[3)] $\int_{Q_j} M_{Q_0}w \, d\mu \le S \int_{Q_j} M_{Q_j}w \, d\mu$ for every $j$.
  \end{enumerate}
\end{lemma}

\begin{remark}
  In our generality we might encounter a situation where some cubes $Q_\alpha^{k,t}$ and $Q_\beta^{k-m,t}$ contain exactly the same points for all $m = 1,2,\ldots,l$
  and some indices $\alpha$ and $\beta = \beta(k)$. Thus, in this lemma we assume that the quantities $\ell(Q_j)$ are maximal. Since $Q_j \subsetneq Q_0$ 
  for all $j$, we know that the quantities $\ell(Q_j)$ are bounded.
\end{remark}

\begin{proof}[Proof of Lemma \ref{lemma:A_infty_wRHI}]
  Notice that $M_{Q_0}w(x) > \lambda$ if and only if $w_Q > \lambda$ for some $Q \ni x$, $\ell(Q) < \ell(Q_0)$. 
  Thus, we know that $\{M_{Q_0}w > \lambda\} = \bigcup_j Q_j$ for some cubes $Q_j$ such that $\ell(Q_j) < \ell(Q_0)$. The properties
  1 and 2 follow immediately.
  
  Let us then prove the property 3. Notice that if $x \in Q_j$, then
  \begin{align*}
    \Q_{Q_0}^x \ = \ \{Q \in \Q_{Q_0}^x \colon \ell(Q) \le \ell(Q_j) \} \cup \{ Q \in \Q_{Q_0}^x \colon \ell(Q_j) < \ell(Q) \le \ell(Q_0)\} \ \eqqcolon \ \mathcal{A} \cup \mathcal{B},
  \end{align*}
  where $\mathcal{A} \subseteq \Q_{Q_j}^x$ and $\mathcal{B}$ is a finite set. Let $Q \in \mathcal{B}$.
  \begin{enumerate}
    \item[i)] If $\ell(Q) < \ell(Q_0)$, then $Q^* \in \Q_{Q_0}$ and $Q_j \subsetneq Q^*$. Thus, by property 2 we have 
              \begin{align*}
                w_Q \le Sw_{Q^*} \le S\lambda \le Sw_{Q_j}.
              \end{align*}
              
    \item[ii)] If $\ell(Q) = \ell(Q_0)$, then $Q \subseteq Q_0^*$ and
               \begin{align*}
                 w_Q \le Sw_{Q_0^*} \le \lambda \le w_{Q_j}.
               \end{align*}
  \end{enumerate}
  Hence, for every $x \in Q_j$ we have
  \begin{align*}
    M_{Q_0}w(x) \, = \, \sup_{Q \in \mathcal{A} \cup \mathcal{B}} w_Q \, \le \, S \cdot \sup_{Q \in \Q_{Q_j}^x} w_Q \, = \, S \cdot M_{Q_j} w(x)
  \end{align*}
  and the property 3 follows.
\end{proof}

The next lemma is a counterpart of \cite[Lemma 2.2]{hytonenperezrela}.

\begin{lemma}
  \label{sharp_lemma}
  Let $w \in A_\infty^\D$ and let $Q_0 \in \D$. Then for any $0 < \varepsilon \le \frac{1}{2S^2K[w]_\infty^\D}$, we have
  \begin{align*}
    \frac{1}{\mu(Q_0)} \int \left(M_{Q_0}w\right)^{1+\varepsilon} d\mu \ \le \ 2S^{1+\varepsilon}[w]_\infty^\D \left( \fint_{Q_0^*} w \, d\mu \right)^{1+\varepsilon}.
  \end{align*}
\end{lemma}

\begin{proof}
  Let first $M \coloneqq M^n$, $n \in \N$, be a bounded localized ``dyadic'' maximal operator: $Mf(x) = \min\{M_{Q_0}f(x), n\}$. Let us define 
  $\Omega_\lambda \coloneqq \{Mw > \lambda\}$ for every $\lambda \ge 0$. Then
  \begin{align*}
    \int \left( Mw \right)^{1+\varepsilon} \, d\mu \ &= \ \int_0^\infty \varepsilon \lambda^{\varepsilon-1} Mw(\Omega_\lambda) \, d\lambda \\
                                                     &\le \ \int_0^{Sw_{Q_0^*}} \varepsilon \lambda^{\varepsilon-1} \left( \int Mw \, d\mu \right) d\lambda + \int_{Sw_{Q_0^*}}^\infty \varepsilon \lambda^{\varepsilon -1} Mw(\Omega_\lambda) \, d\lambda \\
                                                     &\le \ (Sw_{Q_0^*})^\varepsilon [w]_\infty^\D w(Q_0^*) + \int_{Sw_{Q_0^*}}^\infty \varepsilon \lambda^{\varepsilon -1} Mw(\Omega_\lambda) \, d\lambda.
  \end{align*}  
  Let $\lambda \ge Sw_{Q_0^*}$ and let $\{Q_j\}_j$ be as in Lemma \ref{lemma:A_infty_wRHI}. Then
  \begin{align*}
    Mw(\Omega_\lambda) \ \le \ \sum_j \int_{Q_j} Mw \, d\mu
                         &\overset{\ref{lemma:A_infty_wRHI}, 3)}{\le} \ \sum_j S \int M_{Q_j} w \, d\mu \\ 
                         &\le \ \sum_j S[w]_\infty^\D w_{Q_j^*} \mu(Q_j^*) \\
                         &\le \ \sum_j S^2[w]_\infty^\D \lambda \mu(Q_j) 
                       \ \le \ S^2 K[w]_\infty^\D \lambda \mu(\Omega_\lambda)
  \end{align*}
  since $\bigcup_j Q_j = \bigcup_{t=1}^K \bigcup_j Q_j^t$ where $Q_j^t \cap Q_k^t = \emptyset$ for $j \neq k$ for every $t = 1, 2, \ldots, K$.
  Hence,
  \begin{align*}
    \int \left( Mw \right)^{1+\varepsilon} \, d\mu \ &\le \ (Sw_{Q_0^*})^\varepsilon [w]_\infty^\D w(Q_0^{*}) + \varepsilon S^2 K[w]_\infty^\D \int_{Sw_{Q_0^*}}^\infty \lambda^\varepsilon \mu(\Omega_\lambda) \, d\lambda \\
                                                     &\le \ (Sw_{Q_0^{*}})^\varepsilon [w]_\infty^\D w(Q_0^*) + \frac{\varepsilon S^2 K[w]_\infty^\D}{1+\varepsilon} \int (Mw)^{1+\varepsilon} \, d\mu,
  \end{align*}
  and furthermore
  \begin{align*}
    \frac{1}{\mu(Q_0)} \int \left( Mw \right)^{1+\varepsilon} \, d\mu \ \le \ S^{1+\varepsilon}[w]_\infty^\D w_{Q_0^*}^{1+\varepsilon} + \frac{\varepsilon S^2K[w]_\infty^\D}{1+\varepsilon} \frac{1}{\mu(Q_0)}\int (Mw)^{1+\varepsilon} \, d\mu.
  \end{align*}
  Now we can use the boundedness of $M$ and move the last term to the left hand side which gives us the desired inequality for the operator $M$ and for every $0 < \varepsilon \le \frac{1}{2S^2K[w]_\infty^\D}$.
  
  The original claim follows now from the previous case and the monotone convergence theorem as $n \to \infty$.
\end{proof}

\begin{theorem}[Weak Reverse Hölder Inequality for $A_\infty^\D$]
  \label{thm:wRHI_dyadic}
  Let $w \in A_\infty^\D$ and $Q_0 \in \D$. Then 
  \begin{align*}
    \fint_{Q_0} w^{1+\varepsilon} \, d\mu \ \le \ 2S^{1+\varepsilon}\left( \fint_{Q_0^*} w \, d\mu \right)^{1+\varepsilon}
  \end{align*}
  for every $0 < \varepsilon \le \frac{1}{2S^2 K[w]_\infty^\D}$. In particular, $w \in RH_{1+\varepsilon}^\D$ and $[w]_{RH_{1+\varepsilon}}^\D \lesssim 1$.
\end{theorem}

\begin{proof}
  Notice that
  \begin{align*}
    \int_{Q_0} w^{1+\varepsilon} \, d\mu \ \le \ \int (M_{Q_0}w)^\varepsilon w \, d\mu.
  \end{align*}
  Let $\Omega_\lambda$ and $Q_j$ be as in the proof of previous lemma. Then, using similar arguments as earlier, we see that
  \begin{align*}
    \int_{Q_0} (M_{Q_0}w)^\varepsilon w \, d\mu \ &\le \ \int_0^{Sw_{Q_0^*}} \varepsilon\lambda^{\varepsilon-1} w(Q_0^*) \, d\lambda + \int_{Sw_{Q_0^*}}^\infty \varepsilon\lambda^{\varepsilon-1} w(\Omega_\lambda) \, d\lambda \\
                                                  &\le \ (Sw_{Q_0^*})^\varepsilon w(Q_0^*) + \varepsilon \int_{Sw_{Q_0^*}}^\infty \lambda^{\varepsilon-1} \sum_j w(Q_j) \, d\lambda \\
                                                  &\le \ (Sw_{Q_0^*})^\varepsilon w(Q_0^*) + \varepsilon S \int_{Sw_{Q_0^*}}^\infty \lambda^{\varepsilon} \sum_j \mu(Q_j) \, d\lambda \\
                                                  &\le \ (Sw_{Q_0^*})^\varepsilon w(Q_0^*) + \frac{\varepsilon SK}{1+\varepsilon} \int (M_{Q_0}w)^{1+\varepsilon} \, d\mu.
  \end{align*}
  Now averaging over $Q_0$ and using Lemma \ref{sharp_lemma} give us
  \begin{align*}
    \fint_{Q_0} w^{1+\varepsilon} \, d\mu \ &\le \ S^{1+\varepsilon} w_{Q_0^*}^{1+\varepsilon} + \frac{\varepsilon SK}{1+\varepsilon} \frac{1}{\mu(Q_0)} \int (M_{Q_0}w)^{1+\varepsilon} \, d\mu \\
                                            &\le \ S^{1+\varepsilon} w_{Q_0^*}^{1+\varepsilon} + \frac{2\varepsilon SK[w]_\infty^\D \cdot S^{1+\varepsilon}}{1+\varepsilon} w_{Q_0^*}^{1+\varepsilon} \\
                                            &\le \ 2S^{1+\varepsilon} \left( \fint_{Q_0^*} w \, d\mu \right)^{1+\varepsilon},
  \end{align*}
  which is what we wanted.
\end{proof}

\begin{corollary}[Weak Reverse Hölder Inequality for $A_\infty^\text{weak}$]
  Let $w \in A_\infty^\text{weak}$. Then for every $\sigma > \kappa$ there exists a constant $\alpha \coloneqq \alpha(D,\kappa,\sigma)$ such that
  for every $0 < \varepsilon \le \frac{1}{\alpha[w]_\infty^\sigma}$ we have
  \begin{align*}
    \left( \fint_B w^{1+\varepsilon} \, d\mu \right)^{\frac{1}{1+\varepsilon}} \ \lesssim_\sigma \ \fint_{\sigma B} w \, d\mu.
  \end{align*}
  In particular, $w \in RH_q^\text{weak}$ for every $q \in (1,1+\varepsilon]$, $\varepsilon = \varepsilon(w) > 1$, and $[w]_{RH_q}^\sigma \lesssim_\sigma 1$ for $\sigma > \kappa$. 
\end{corollary}

\begin{proof}
  The claim follows from previous results (Lemma \ref{lemma:A_infty_dyadic}, Theorem \ref{thm:wRHI_dyadic}, Lemma \ref{lemma:dyadicRHq} and Lemma \ref{lemma:dyadic_RH_q_constants}).
\end{proof}

We can now give an alternative characterization to the weak $A_\infty$ class using the previous result:

\begin{theorem}
  \label{thm:class_characterization}
  $w \in A_\infty^\text{weak}$ if and only if $w \in RH_q^\text{weak}$ for some $q > 1$.
\end{theorem}

\begin{proof}
  Let $w \in RH_q^\sigma$. Then by Hölder's inequality and the $L^q$-boundedness of the Hardy-Littlewood maximal function we have
  \begin{align*}
    \fint_B M(1_B w) \, d\mu \ \le \ \left( \fint_B M(1_B w)^q \, d\mu \right)^{1/q} \ \lesssim_q \ \left( \fint_B w^q \, d\mu \right)^{1/q} \ \le \ [w]_{RH_q}^\sigma \fint_{\sigma B} w \, d\mu,
  \end{align*}
  which gives us the claim.
\end{proof}

\begin{remark}
  \label{remark:strong_classes}
  Recall that Theorem \ref{thm:class_characterization} holds also for the strong weight classes in the Euclidean setting \cite{stein}.
  Although choosing $\sigma = 1$ in the previous proof gives us the ``$RH_q \Rightarrow A_\infty$'' type result,
  the other direction of the theorem does not hold for the strong weight classes in our setting. We will prove this in Section \ref{section:failure_gehring}.
\end{remark}

\section{Gehring's lemma for weak $RH_q$ classes}
\label{section:gehring_lemma}

The self-improving properties of different Reverse Hölder classes have been explored thoroughly by different authors during the last decades. The original Euclidean
``$RH_q \Rightarrow RH_{q+\varepsilon}$'' type result by F. Gehring \cite[Lemma 3]{gehring} has been generalized to different forms by e.g. 
A. Zatorska-Goldstein \cite[Theorem 3.3]{zatorska-goldstein} and O. Maasalo \cite[Theorem 3.1]{maasalo} (see also a recent preprint by F. Bernicot, T. Coulhon and D. Frey 
\cite[Theorem C.1]{bernicotetall}). We will show later in Section \ref{section:failure_gehring} that a strong ``$RH_q \Rightarrow RH_{q+\varepsilon}$'' Gehring's lemma 
does not hold in spaces of homogeneous type. It does, however, hold in doubling metric measure spaces that satisfy some additional geometrical properties 
(such as the $\alpha$-annular decay property \cite[Corollary 3.2]{maasalo}).

In this section we will show that the $RH_q^\text{weak}$ classes have a self-improving property in spaces of homogeneous type. The metric version of 
our result is actually a special case of \cite[Theorem 3.3]{zatorska-goldstein} (with the choices $f \equiv 0$, $\theta = 0$) but our proof 
is drastically different and considerably shorter: we do not need to rely on any additional decompositions since we can use the results we proved in Sections 
\ref{section:weak_weight_classes}, \ref{section:dyadic_weight_classes} and \ref{section:A_infty_wRHI}.

\begin{lemma}
  \label{lemma:RHD-A_inftyD}
  If $w \in RH_q^\D$, then $w^q \in A_\infty^\D$.
\end{lemma}

\begin{proof}[Proof of Lemma \ref{lemma:RHD-A_inftyD}]
  Let $Q_0 \in \D$ be fixed. Then for every $x \in \bigcup_{Q \in \Q_{Q_0}} Q$ we have
  \begin{align*}
    M_{Q_0}w^q(x) \ = \ \sup_{Q \in \Q_{Q_0}^x} (w^q)_Q 
                  \ \le \ ([w]_{RH_q}^\D)^q \cdot \sup_{Q \in \Q_{Q_0}^x} (w_{Q^*})^q 
                  \ \le \ ([w]_{RH_q}^\D)^q \cdot (M_{Q_0^*}w(x))^q.
  \end{align*}
  Thus,
  \begin{align*}
    \int M_{Q_0}w^q \, d\mu \ \le \ ([w]_{RH_q}^\D)^q \int (M_{Q_0^*}w)^q \, d\mu
                            \ &= \ ([w]_{RH_q}^\D)^q \int_{Q_0} (M_{Q_0^*} (1_{Q_0^{**}} w))^q \, d\mu \\
                              &\lesssim \ ([w]_{RH_q}^\D)^q \int_{Q_0} (M(1_{Q_0^{**}} w))^q \, d\mu \\
                              &\le \ C_q ([w]_{RH_q}^\D)^q w^q(Q_0^{**}),
  \end{align*}
  where the last step used the $L^q$-boundedness of the Hardy-Littlewood maximal function. Thus, $w^q \in A_\infty^\D$ by Remark \ref{remark:dyadic_A_infty}.
\end{proof}

\begin{proposition}
  If $w \in RH_q^\text{weak}$, then there exists a constant $\beta \coloneqq \beta(D,\kappa,q)$ such that we have $w \in RH_{q+\varepsilon}^\text{weak}$ for every $0 \le \varepsilon \le \frac{1}{\beta \cdot ([w]_{RH_q}^\D)^q}$.
\end{proposition}

\begin{proof}
  Let $\sigma > \kappa$ so large that $Q^{**}_B \subseteq \sigma B$ for every ball $B$ and let $w \in RH_q^\sigma$. Then by Lemma \ref{lemma:dyadicRHq}
  and Lemma \ref{lemma:RHD-A_inftyD} we know that $w^q \in A_\infty^\D$. Thus, by Theorem \ref{thm:wRHI_dyadic}, we know that 
  for every $0 \le \tilde{\varepsilon} \le \frac{1}{2S^2 K[w^q]_\infty^\D}$ and for every $r \in [1,1+\tilde{\varepsilon}]$ we have
  $w^q \in RH_r^\D$. Hence, for every ball $B$ it holds that
  \begin{align*}
    \left( \fint_B w^{qr} \, d\mu \right)^{1/qr} \ \lesssim_{q,r} \ \left( \fint_{Q_B} w^{qr} \, d\mu \right)^{1/qr} 
                                                 \ &\le \ ([w^q]_{RH_r}^\D)^{1/q} \left( \fint_{Q_B^*} w^q \, d\mu \right)^r \\
                                                   &\le \ ([w^q]_{RH_r}^\D)^{1/q} [w]_{RH_q}^\D \fint_{Q_B^{**}} w \, d\mu \\
                                                   &\lesssim \ ([w^q]_{RH_r}^\D)^{1/q} [w]_{RH_q}^\D \fint_{\sigma B} w \, d\mu.
  \end{align*}
  By the proof of Lemma \ref{lemma:RHD-A_inftyD}, we know that $[w^q]_\infty^\D \le C_q ([w]_{RH_q}^\D)^q$. Thus, there exists a constant
  $\beta \coloneqq \beta(D,\kappa,q)$ such that $q \le qr \le q+\frac{1}{\beta\cdot([w]_{RH_q}^\D)^q}$. In particular, we have 
  $w \in RH_{q + \varepsilon}^\text{weak}$ for every $0 \le \varepsilon \le \frac{1}{\beta\cdot([w]_{RH_q}^\D)^q}$.
\end{proof}

\section{Failure of strong results}
\label{section:failure_gehring}

In this section we show that the results we presented in earlier sections are essentially the best kind we can hope for. In other words, 
we will show that both an ``$A_\infty \Rightarrow \text{RHI}$'' type theorem and a strong ``$RH_q \Rightarrow RH_{q+\varepsilon}$'' 
type Gehring's lemma are out of reach in general spaces of homogeneous type, even in a purely metric case. We do this by constructing a doubling metric
measure space in which some functions fail the properties we mentioned.

\subsection{Construction of the space and some of the functions}

Consider $\R^2$ with the $\ell^{\infty}$ metric, so that balls are actually squares. We define $X$ as a subset of $\R^2$ consisting of an infinite line with finite 
line-segments attached. Let
\begin{equation*}
  A \coloneqq \{(u,0):u\in\R\},\quad
  U \coloneqq \{(u,\frac12u):u\in(0,1]\},\quad
  V \coloneqq \{(1,v):v\in[\frac12,1]\},\quad
  W \coloneqq U\cup V.
\end{equation*}
We take $X \coloneqq A\cup\bigcup_{j\in\N}W_j$ with the $\ell^{\infty}$ metric and the arc-length measure, where $W_j:=W+(10j,0) \eqqcolon U_j \cup V_j$. This is an Ahlfors 
$1$-regular metric measure space.

The reason for using this particular space is that for suitable functions we only need to test the properties we mentioned earlier for couple different types of balls. We 
will use functions constructed in the following way. Let $\varepsilon_j \to 0^+$, $\varepsilon_j \le 1$, and let $h$ be a positive function defined on the interval $(0,1)$.
Let us then set $g(t) \coloneqq \max\{h(t),1\}$ and define the function $f \coloneqq f_h \colon X \to \R_+$ by setting
\begin{align*}
  f(x) \coloneqq \left\{ \begin{array}{cl}
                           1, & \text{if } x \in A \\
                           \varepsilon_j, & \text{if } x \in V_j \\
                           \min\{1,\varepsilon_j g(u)\}, & \text{if } x = (10j + u,\frac{1}{2}u) \in U_j
                         \end{array} \right. .
\end{align*}
Notice that $f \le 1$ everywhere.

\subsection{Failure of strong Reverse Hölder property of $A_\infty$ weights}
\label{subsection:failure_of_RHI}

Let $h(t) \coloneqq t^{-1} \log^{-3}(e/t)$. Then $h \in L^1(0,1)$ but $h \notin L^p(0,1)$ for any $p > 1$. 
We will show that now $f = f_h \in A_\infty$ but $f \notin RH_p$ for any $p > 1$.

Let us test the $A_\infty$ condition for different squares $Q \coloneqq Q(x,r) \coloneqq (x_1-r,x_1+r) \times (x_2-r,x_2+r)$. Let use first assume that $Q \cap A \neq \emptyset$. Then 
$Q \cap A$ is a line-segment of length $\mu(Q \cap A) = 2r \ge c\mu(Q)$. Thus,
\begin{align*}
  \int_Q M(1_Q f) \, d\mu \ \le \ \mu(Q) \le \frac{1}{c} \mu(Q \cap A) \ \le \ \frac{1}{c} f(Q),
\end{align*}
since $f(x) \le 1$ for every $x \in X$ and $f = 1$ on $A$.

From now on, we consider only squares which do not meet $A$. Then $x \in W_j$ for some $j$ and without loss of generality we may consider $j = 0$ and $\varepsilon_0 = \varepsilon$.
Let first $x = (u,\frac{1}{2}u) \in U$. Then $r \le \frac{1}{2}u$ and thus $(u+r)/(u-r) \le 3$. Now
\begin{align*}
  \sup_{x \in Q} f(x) \ &= \ \min\left\{1, \sup_{t \in (u-r,\min\{u+r,1\})} \varepsilon g(t)\right\}, \\
  \inf_{x \in Q} f(x) \ &= \ \min\left\{1, \inf_{t \in (u-r,\min\{u+r,1\})} \varepsilon g(t)\right\}.
\end{align*}
In particular, $\sup_{x \in Q} f(x) / \inf_{x \in Q} f(x) \le 3$. Hence,
\begin{align*}
  \int_Q M(1_Q f) \, d\mu \ \le \ \mu(Q) \cdot \sup_{x \in Q} f(x) \ \le \ 3\mu(Q) \cdot \inf_{x \in Q} f(x) \ \le \ 3 f(Q).
\end{align*}

Finally, let $ x = (1,v) \in V$. Since $Q \cap A = \emptyset$, $r \le v \le 1$. If $r \le \frac{1}{2}$, then $\sup_Q f$ and $\inf_Q f$ have a ratio of at most $2$ and
the previous consideration applies. Let then $r \in (\frac{1}{2},1]$ and $s_0 \in (0,1)$ be the point such that $h(s_0) = 1$ and $u_\varepsilon \in (0,1)$ the point 
such that $h(u_\varepsilon) = 1/\varepsilon$. We may assume that $1-r \le u_\varepsilon$ since the other cases can be generalized 
easily from this case. Then
\begin{align*}
  \int_Q f \, d\mu \ &\eqsim \ \int_{1-r}^1 \min\{1, \max\{\varepsilon,\varepsilon h(t)\}\} \, dt + \int_{1/2}^1 \varepsilon \, dt \\
                     &= \        \int_{1-r}^{u_\varepsilon} 1 \, dt + \int_{u_\varepsilon}^{s_0} \varepsilon h(t) \, dt + \int_{s_0}^1 \varepsilon \, dt + \int_{1/2}^1 \varepsilon \, dt \\
                     &\eqqcolon \ I_1 + I_2 + I_3 + I_4.
\end{align*}
Notice that the function $t \mapsto \max\{\varepsilon,\varepsilon h(t)\}$ is descending and $h,Mh \in L^1(0,1)$. Thus, by elementary calculations, we have 
\begin{align*}
  \int_Q M(1_Qf) \, d\mu \ \lesssim \ J_1 + J_2 + J_3 + J_4
\end{align*}
for such $J_i$ that $J_i \le C \cdot I_i$ for some constant $C \ge 1$ independent of $\varepsilon$ and $Q$ and every
$i = 1,2,3,4$. In particular, $\int_Q M(1_Qf) \, d\mu \lesssim \int_Q f \, d\mu$ and $f \in A_\infty$.

Let us then show that $f \notin RH_p$ for any $p > 1$. Consider the particular square $Q_j$ of centre $(10j+1,1)$ and radius $1$. Thus in fact $Q_j=W_j=U_j\cup V_j$. Then
\begin{equation*}
  \fint_{Q_j} f^p\ud\mu\geq c\int_{U_j} f^p\ud\mu=c\int_0^1\eps_j^p\min\{\frac{1}{\eps_j},g\}^p\ud u,
\end{equation*}
whereas
\begin{equation*}
  \fint_{Q_j} f\ud\mu\leq C\Big(\int_0^1 \eps_j g\ud u+\int_{1/2}^1\eps_j\ud v\Big)\leq C\eps_j, 
\end{equation*}
so that
\begin{equation*}
  \Big(\fint_{Q_j} f^p\ud\mu\Big)^{1/p}\Big/\Big(\fint_{Q_j} f\ud\mu\Big)
  \geq c\Big(\int_0^1\min\{\frac{1}{\eps_j},g\}^p\ud u\Big)^{1/p}.
\end{equation*}
As $\eps_j\to 0^+$, the right side tends to $c\big(\int_0^1 g^p\ud u\big)^{1/p}=\infty$ by monotone convergence. Thus, $f \notin RH_p$ for any $p > 1$.

\subsection{Failure of strong Gehring's lemma}

Let $h(t):=t^{-\alpha}\log^{-1}(e/t)$ with some $0<\alpha<1$. Then $h\in L^p(0,1)$ if and only if $p\leq 1/\alpha$. This function obviously does not belong to $RH_{1/\alpha}$ on $[0,1]$ with the Lebesgue measure, as this would contradict the classical Gehring lemma (and it is also easy to check this directly).
We claim that now $f = f_h \in RH_p$ if and only if $p\leq 1/\alpha$.

Our strategy is the same as in previous subsection and thus, we only need to test the $RH_p$ condition for three different types of squares $Q \coloneqq Q(x,r)$.
The cases $Q \cap A \neq \emptyset$, $x \in U_j$ and $x \in V_j$, $r \le 1/2$, can be checked similarly as earlier.

Let $x=(1,v)\in V$, $Q\cap A=\varnothing$ and $r \in (1/2,1]$. Now $Q$ contains all of $V$ and $Q\subseteq U\cup V$. Hence, for $p\leq 1/\alpha$ we have
\begin{align*}
  \Big(\fint_Q f^p\ud\mu\Big)^{1/p}
  &\leq C\Big(\int_Q f^p\ud\mu\Big)^{1/p}\leq C\Big(\int_{U\cup V}f^p\ud\mu\Big)^{1/p}
  \leq C\Big(\int_0^1 (\eps g)^p\ud u+\int_{1/2}^1\eps^p\ud v\Big)^{1/p} \\
  &\leq C\eps\leq C\int_{1/2}^1\eps\ud v
   =C\int_V f\ud\mu\leq C\int_Q f\ud\mu\leq C\fint_Q f\ud\mu,
\end{align*}
where we interchanged twice between $\int_Q$ and $\fint_Q$ by the fact that $r\in(\frac12,1]$ and the Ahlfors-regularity of $\mu$, and we used the fact that $\int_0^1 g^p\ud u\leq C$ for $p\leq 1/\alpha$. 
Thus, $f \in RH_p$ for $p\leq 1/\alpha$.

The proof for the failure of the $RH_p$ property for $p > 1/\alpha$ is almost identical to the proof of the failure 
of the $RH_p$ property for $p > 1$ in the previous subsection.

\begin{remark}
By using geometrically decreasing copies of $W$ instead of simple translates, we could have arranged the counterexample inside a compact set, if desired.

Moreover, if we interpreted the $RH_p$ property in an extended sense, by still requiring that $(\fint_B f^p\ud\mu)^{1/p}\leq C\fint_B f\ud\mu$ hold for all balls, but possibly with both sides equal to $\infty$, then we could simply take $X=A\cup W$ and $f(x)=\infty$ for $x\in A$, $f(u,\frac12u)=g(u)$ for $(u,\frac12u)\in U$, and $f(x)=1$ for $x\in V$.
\end{remark}

\section{Equivalence of different definitions}
\label{section:equivalent_definitions}

Like we mentioned earlier, there are numerous different definitions for the $A_\infty$ class in the Euclidean setting but 
these definitions are not equivalent in general spaces of homogeneous type. Hence, in this context, it is important 
to be specific about which definition is being used for the $A_\infty$ class. However, some weakened definitions are 
equivalent also in our generality. Previously, the following definition for weak $A_\infty$ weights has appeared in some articles related to analysis 
in $\R^n$ for $\sigma = 2$ \cite{hofmannmartell, sawyer}:
\begin{defin}
  \label{definition:old_A_infty_class}
  Let $\sigma \ge 1$ and let $w$ be a weight. We denote $w \in \A_\infty^\sigma$ if there exists constants $C > 0$ and $p \ge 1$ such that 
  for every ball $B$ and every measurable set $E \subseteq B$ we have
  \begin{align*}
    \frac{w(E)}{w(\sigma B)} \le C \left( \frac{\mu(E)}{\mu(B)} \right)^{1 / p}.
  \end{align*}
\end{defin}
It is straightforward to show that the class $\A_\infty^\sigma$ contains the same functions as the class $A_\infty^\sigma$ if $\sigma > \kappa$. We show this by using 
the techniques from the proof of \cite[Chapter I, Lemma 12]{strombergtorchinsky} to show that $w \in \A_\infty^\sigma$ if and only if 
$w \in RH_q^\sigma$ for some $q > 1$, which gives us the claim by Theorem \ref{thm:class_characterization}. 
However, in the case $\sigma = 1$, this result does not hold. Since $w \in \A_\infty \coloneqq \A_\infty^1$ if and only if $w \in RH_q$ for some $q > 1$ by 
\cite[Chapter I, Lemma 12]{strombergtorchinsky}, the Fujii-Wilson $A_\infty$ condition is strictly weaker than the 
$\A_\infty$ condition by Remark \ref{remark:strong_classes} and Section \ref{subsection:failure_of_RHI} 

\begin{lemma}
  \label{lemma:same_classes}
  $\A_\infty^\sigma = A_\infty^\sigma = \A_\infty^{\sigma'}$ for every $\sigma, \sigma' > \kappa$.
\end{lemma}

\begin{proof}
  Suppose that $w \in A_\infty^\sigma$. Let us fix a ball $B$ and a measurable set $F \subseteq B$. Now, by Theorem \ref{thm:class_characterization}, we know that $w \in RH_q^\sigma$ for some $q > 1$. Thus,
  \begin{align*}
    w(F) \, \le \, \mu(B) \left( \fint_B w^q \, d\mu \right)^{1/q} \left( \fint_B 1_F \, d\mu \right)^{1/q'}
         \, \le \, [w]_{RH_q}^\sigma \frac{\mu(B)}{\mu(\sigma B)} w(\sigma B) \left( \frac{\mu(F)}{\mu(B)} \right)^{1/q'}.
  \end{align*}
  In particular, $w \in \A_\infty^\sigma$.
  
  Suppose then that $w \in \A_\infty^\sigma$. Let us fix a ball $B$ and denote 
  $E_\lambda \coloneqq \{x \in B \colon w(x) > \lambda\}$ for every $\lambda \ge 0$. Then we have
  \begin{align*}
    \mu(E_\lambda) \le \frac{1}{\lambda} w(E_\lambda) \le \frac{C}{\lambda} \left( \frac{\mu(E_\lambda)}{\mu(B)} \right)^{1/p} w(\sigma B),
  \end{align*}
  which gives us
  \begin{align}
    \label{e_lambda_estimate} \mu(E_\lambda) \le \min\left\{ \mu(B), C^{p'} \frac{w(\sigma B)^{p'}}{\lambda^{p'} \mu(B)^{1 / (p-1)}} \right\}.
  \end{align}
  Suppose then that $r \in [1,1 + 1/p)$ and write
  \begin{align*}
    \int_B w^r \, d\mu \, = \, r \int_0^{C w(\sigma B) / \mu(B)} \lambda^{r-1} \mu(E_\lambda) \, d\lambda
                           + r \int_{C w(\sigma B) / \mu(B)}^\infty \lambda^{r-1} \mu(E_\lambda) \, d\lambda
                       \, \eqqcolon \, I + II.
  \end{align*}
  Since $r -p' < 0$, the estimate \eqref{e_lambda_estimate} gives us
  \begin{align*}
     I \le C^r w(\sigma B)^r \mu(B)^{1-r}
  \end{align*}
  and
  \begin{align*}
    II \le C^{p'} \cdot \frac{w(\sigma B)^{p'}}{\mu(B)^{p' / p}} \cdot \frac{r}{r - p'} \cdot \left(- \frac{C^{r - p'} \omega(\sigma B)^{r - p'}}{\mu(B)^{r-p'}} \right) 
       = C^r \frac{r}{p' - r} w(\sigma B)^{r} \mu(B)^{1-r}.
  \end{align*}
  Thus, since $p' - r \ge 1 / (p(p-1))$, we have
  \begin{align*}
    I + II \le \left( \frac{r}{p' -r} + 1 \right) C^r w(\sigma B)^r \mu(B)^{1-r}
           \le p^2 C^r w(\sigma B)^r \mu(B)^{1-r}.
  \end{align*}
  In particular,
  \begin{align*}
    \left( \fint_B w^r \, d\mu \right)^{1/r} \le p^{2 / r} C \frac{1}{\mu(B)} \int_{\sigma B} w \, d\mu \le p^2 C D^{\log_2 \sigma + 1} \fint_{\sigma B} w \, d\mu,
  \end{align*}
  so by Theorem \ref{thm:class_characterization} we know that $w \in A_\infty^\sigma$.
  
  Hence, $A_\infty^\sigma = \A_\infty^\sigma$ for every $\sigma > \kappa$ and the claim follows from Theorem \ref{thm:A-infty-classes}.
\end{proof}

Thus, it is natural to set $w \in \A_\infty^\text{weak}$ if $w \in \A_\infty^\sigma$ for some $\sigma > \kappa$. Lemma \ref{lemma:same_classes} gives us now the following expansion 
of Theorem \ref{thm:class_characterization}:

\begin{theorem}
  The following three conditions are equivalent:
  \begin{enumerate}
    \item[1)] $w \in A_\infty^\text{weak}$
    \item[2)] $w \in \A_\infty^\text{weak}$
    \item[3)] $w \in RH_q^\text{weak}$ for some $q > 1$.
  \end{enumerate}
\end{theorem}

\end{document}